\documentclass[12pt,twoside]{amsart}
\usepackage{amssymb}
\usepackage{verbatim}
\usepackage{amsmath}
\usepackage{bm}

\usepackage{ulem}
\usepackage[usenames, dvipsnames]{color}
\usepackage{a4wide}
\usepackage[latin1]{inputenc}
\usepackage[T1]{fontenc}
\usepackage{times}
\usepackage{amssymb,latexsym}
\usepackage{enumerate}
\usepackage[stable]{footmisc}
\usepackage{color}
\usepackage[colorlinks,linkcolor=black,citecolor=black,urlcolor=black]{hyperref}

\newcommand{\R}{\ensuremath{\mathbb{R}}}
\newcommand{\D}{\ensuremath{\mathbb{D}}}
\newcommand{\C}{\ensuremath{\mathbb{C}}}

\makeatletter
\newcommand{\sumprime}{\if@display\sideset{}{'}\sum%
            \else\sum'\fi}
\makeatother

\begin{document}

\numberwithin{equation}{section}

\newtheorem{theorem}{Theorem}[section]
\newtheorem{proposition}[theorem]{Proposition}
\newtheorem{conjecture}[theorem]{Conjecture}
\def\theconjecture{\unskip}
\newtheorem{corollary}[theorem]{Corollary}
\newtheorem{lemma}[theorem]{Lemma}
\newtheorem{observation}[theorem]{Observation}
\newtheorem{definition}{Definition}
\numberwithin{definition}{section} 
\newtheorem{remark}{Remark}
\def\theremark{\unskip}
\newtheorem{kl}{Key Lemma}
\def\thekl{\unskip}
\newtheorem{question}{Question}
\def\thequestion{\unskip}
\newtheorem{example}{Example}
\def\theexample{\unskip}
\newtheorem{problem}{Problem}

\thanks{The first author was supported by NSERC (Canada) grant RGPIN-2016-04107. The  second author was supported by Natural Science Foundation of Guangdong Province,  No.  2023A1515030017  and National Key R\&D Program of China, No. 2024YFA1015200.}

\address [P. M. Gauthier] {D\'epartement de math\'ematiques et de statistique, Universit\'e de Montr\'eal, Montr\'eal, Qu\'ebec, 	H3C3J7, Canada}
\email{paul.m.gauthier@umontreal.ca, }

\address [Jujie Wu] {School of Mathematics (Zhuhai), Sun Yat-Sen University, Zhuhai, Guangdong,  519082, P. R.}
\email{ wujj86@mail.sysu.edu.cn}

\title{Holomorphic motion, rational approximation and  an equivalent   formulation of the Riemann Hypothesis}
\author{P. M. Gauthier and Jujie Wu$^{*}$}

\date{}
\begin{abstract} 
A compact subset $K$ of the complex plane $\C$ is a set of polynomial (respectively rational) approximation if $P(K)=A(K)$ (respectively  $R(K)=A(K)$), where $P(K)$ (respectively $R(K)$) is the family of functions on $K$ which are uniform limits of polynomials (respectively  rational functions, having no poles on $K$) and $A(K)$ is the family of continuous functions on $K,$ which are holomorphic on the interior of $K.$ In the class of compact sets, the property of being a set of polynomial approximation is easily seen to be invariant under holomorphic motion.  We show that this is no longer the case for  rational approximation. Secondly, we show that the Riemann Hypothesis holds if and only if a certain map is a holomorphic motion. 

    \bigskip
  \noindent{{\sc Mathematics Subject Classification} (2020): 30E10, 37F44  }
  
  \smallskip
  
  \noindent{{\sc Keywords}: Rational approximation, quasiconformal mapping, holomorphic motion, Riemann Hypothesis} 
 \end{abstract}

\maketitle

\tableofcontents

\section{Introduction}

In the preface of  the book \cite{AIM} by Kari Astala, Tadeusz Iwaniec and Gaven Martin, we find the following statement. 
"It is a simply amazing fact that the mathematics that underpins the geometry, structure and dimension of such concepts as Julia sets and limit sets of Kleinian groups, the spaces of moduli of Riemann surfaces, conformal dynamical systems and so forth is the very same as that which underpins existence, regularity, singular set structure and so forth for precisely the most important class of differential equations one meets in physical applications, namely, second-order divergence-type equations.
All these subjects are inextricably linked in two dimensions by the theory of quasiconformal mappings."  In the present paper, we shall display an important connection between holomorphic motion,  quasiconformal mappings, complex dynamics (in particular the dimension of Julia sets) and rational approximation.

\begin{definition}
Let $X$ be a complex manifold. A {\it holomorphic family} of rational maps $f_\lambda(z)$ over $X$ is a holomorphic map $X\times\overline\C\to\overline\C,$ given by $(\lambda,z)\mapsto f_\lambda(z).$
\end{definition}
We shall be concerned with two special cases, firstly, when $X$ is the complement $\overline\C\setminus M$ of the Mandelbrot set $M$ in order to prove our main result regarding the non-invariance of rational approximation under holomorphic motion and secondly, when $X$ is the strip  $\{\lambda: 1/2<\Re\lambda<3/2\},$ in order to give a  statement in terms of holomorphic motion which is equivalent to the Riemann Hypothesis.

The concept of holomorphic motion was first introduced by  Ma$\tilde{\text{n}}\acute{\text{e}}$-Sad-Sullivan \cite{MSS} and was generalized as follows
 (see for example \cite[p. 126]{M}).
\begin{definition}
Let $(W, \lambda_0)$ be a pointed complex manifold,  that is, a manifold with a distinguished point,   and 
denote by $\overline\C=\C\cup\{\infty\}$ the extended complex plane (the Riemann sphere). 	Let $E$ be a subset of $\overline\C$.  A holomorphic motion of $E$ in $\overline\C,$ parametrized by $W$ with base point $\lambda_0$,
is a map
$$
f(\lambda, z):W\times E \rightarrow  \overline\C,
$$
which enjoys the following properties: 
\begin{enumerate}
		\item  $f(\lambda_0,z)=z$ for all $z\in E$;
		\item For every fixed $\lambda\in W$, 
		$z \rightarrow f(\lambda,z)= f_\lambda (z)$ is an injection  on $E;$
		\item For
		every fixed $z\in E$, $\lambda \rightarrow f(\lambda, z)$ is a holomorphic map  on $\D.$
	\end{enumerate}
If $E$ is a subset of $\C,$ a holomorphic motion of $E$ in $\C$ is a map 
$$
f(\lambda, z):W\times E \rightarrow  \C
$$ 
enjoying the same three properties.  
\end{definition}

Usually, one takes $W$ to be the unit disc $\D.$  Note especially that,  in the definition of  holomorphic motion, there is no assumption regarding the  set $E$ and, for each $\lambda\in W,$  the only restriction on the  function $f(\lambda,\cdot):E\to \C$ is that it be injective. If $f$ is a holomorphic motion of $E$ in $\C,$ then it can also be considered as a holomorphic motion of $E$ in $\overline\C$ and 
if $f$ is a holomorphic motion of $E$ in $\overline\C,$ such that $f\left(\D\times(E\cap\C)\right)\subset\C,$ then the restriction of $f$ to $\D\times(E\cap\C)$ is a holomorphic motion of $E\cap\C$ in $\C.$

A deep connection between quasiconformal maps and holomorphic motions is the following remarkable $\lambda$-\/ {\it Lemma\/}, proved partially by  Ma$\tilde{\text{n}}\acute{\text{e}}$-Sad-Sullivan \cite{MSS},
Sullivan-Thurston \cite{SullivanThurston}, Bers-Royden \cite{BersRoyden}, completely by S{\l}odkowski \cite {Slodkowski} in 1991, states that the
holomorphic dependence of the injections on the parameter implies better regularity, namely, a holomorphic motion extends to a holomorphic family of quasiconformal maps
of the whole (extended) complex plane.  For a nice survey of the $\lambda$-Lemma, see also  \cite{GJW2010}.

\begin{theorem} (\cite{MSS, Slodkowski}) \label{th:lambda}
	Every holomorphic motion $f(\lambda, z):\D \times
	E\rightarrow\overline\C$ of $E$  admits an extension to a holomorphic motion $F(\lambda,z):\D\times \overline\C\rightarrow
 \overline\C,$ satisfying
	\begin{enumerate}
		\item Every $F(\lambda,\cdot):\overline \C \rightarrow  \overline\C $ is a
		quasiconformal self-homeomorphism of dilatation not exceeding $\frac{1+|\lambda|}{1-|\lambda|}$ and $F$ is jointly continuous;
		\item If $f(\lambda,z): \D\times  E\to\C,$   we  may take $F(\lambda,z):\D\times\C\to\C$  and $F$ is jointly H\"older continuous in $(\lambda,z)$.
	\end{enumerate}
\end{theorem}

 We recall that a holomorphic map from $\D\subset\C$ to $\overline\C$ is just a meromorphic function on $\D.$ 
Also, we note that the $\lambda$-Lemma  is trivial in case $E$ is a singleton $E=\{e\}.$ Indeed, in this case $f$ has the form $f(\lambda,e)=e$ and we can put $F(\lambda,z)=z.$

The extensions in the $\lambda$-Lemma, show that
in extending a holomorphic motion $f$ of a set $E\subset\C$ to a holomorphic motion of $\overline\C,$ we may  firstly extend to holomorphic motion $F$ of $\C$ and then further extended to a holomorphic motion $\overline F$ of $\overline\C.$ In such a way that the restriction of $\overline F$ to $\D\times E$ is the original holomorphic motion $f.$ 
We thus recuperate all holomorphic motions in $\C.$ Here, we  are interested in holomorphic motions in $\C.$   Thus, $f:\D\times E\to\C.$

Although holomorphic motion $f:\D\times  E\to\C$ extends as a function of the second variable,  it is easy to give an example of a holomorphic motion of the disc $\D,$ which does not extend continuously as a function of the first variable.  For example,   the holomorphic motion $f(\lambda,z)=\sin\big(\lambda/(\lambda-1)\big)+z$ of the unit disc $\D,$ considered as a function of the first variable,   does not extend continuously to the point $\lambda=1$ on the boundary of $\D.$ 

  In sections 2, 3, and 4,  we present some  background material. Readers familiar with these matters may skip to  
Section 5,  where we  prove the following (main) result, showing that rational approximation is  non-invariant under holomorphic motion. 

\begin{theorem}
There is a compact set $K\subset\C,$  a holomorphic motion of $K,$  $H:\C\times K\to\C$ and a $\lambda\in\D,$  such that 
$$
	R(K)=A(K) \quad \mbox{but} \quad R(H_\lambda(K))\not=A(H_\lambda(K)). 
$$

\end{theorem}
Our second result, reflecting the surprising nature of holomorphic motion,  is a  weird criterion for a meromorphic function to be zero-free.  In particular, this gives  that the Riemann Hypothesis holds if and only if a certain simple mapping, closely related to the Riemann Zeta function,  is a holomorphic motion.

\section{Quasiconformal mappings} \label{sectionquasi}

In the present section, we are concerned with the connection between quasiconformal mappings and holomorphic motion and also the non-invariance of dimension under quasiconformal mappings.

 We first review the basic facts about quasiconformal mappings, following the classical book of Ahlfors \cite{AhlforsQuasi}  (see also \cite{LV}).  
Suppose $f$ is ACL in a domain $D\subset\C.$  That is, $f(x,y)$ is absolutely continuous on a.e. horizontal or vertical line in $D.$ Then, $f_z$ and $f_{\overline z}$ are defined a.e. in $D.$  
 
It is assumed that $f$ has non-zero Jacobian and is orientation-preserving. This ensures
that $f_z\not=0$ and, moreover, that $|f_z|>|f_{\overline z}|.$ The {\it dilatation} of $f$ is $D_f=(|f_z| + |f_{\overline{z}}|)/( |f_z|  - |f_{\overline{z}}|).$
Note that it is often more convenient to consider the  {\it complex dilatation\/}  $f_{\overline{z}}/f_z$  of $f$   with $d_f = |f_{\overline{z}}/f_z|$. 
The quantities $d_f$ and $D_f$ are related as follows.  
$$
D_f   = \frac{ 1+d_f}{1-d_f} ,   \ \ \ \ d_f =  \frac{ D_f -1}{D_f +1}.
$$
Note that $D_f\ge 1$ and $d_f\le 1.$ 

A homeomorphism $f$ of domain $D\subset {\mathbb C}$ is called $K$-quasiconformal  if
\begin{enumerate}
	\item $f$ is ACL in $D$;
	
	\item $D_f(z)\le K<+\infty$.

\end{enumerate}

The condition $D_f \leq K$ is equivalent to $d_f\leq k =(K-1)/(K+1)$.  The mapping is conformal at $z$ if and only if $D_f=1, d_f =0.$

Quasiconformal  mappings enjoy the following properties:
\begin{enumerate}
\item (Existence Theorem). For every measurable,  compactly supported function $\mu$ in ${\mathbb C}$ with $\|\mu\|_\infty<1$, there exists a quasiconformal mapping of ${\mathbb C}$ satisfying the following Beltrami equation
    $$
    f_{\bar{z}}=\mu f_z,\ \ \ f(z)/z\rightarrow 1\ (z\rightarrow \infty).
    $$
    
    \item   If $f:D\rightarrow \Omega$ is a $K$-quasiconformal  mapping between plane domains,  then $f^{-1} : \Omega \rightarrow D$  is also a $K$-quasiconformal  mapping.

\item The composition of a $K_1$-quasiconformal  and a $K_2$-quasiconformal  mapping is $K_1K_2$-quasiconformal 
\item If $f$ is a $K$-quasiconformal  mapping of the plane fixing $\infty$, then
$$
C_K^{-1}|z-z'|^K \le |f(z)-f(z')|\le C_K|z-z'|^{1/K}.
$$
\end{enumerate}

It follows from (3)  that a $K$-quasiconformal  mapping  composed with a conformal mapping is again a $K$-quasiconformal mapping.   This allows us to define a homeomorphism $f:X\to Y$  between two Riemann surfaces to be a $K$-quasiconformal  mapping, if it is $K$-quasiconformal  in local coordinates.  In particular, if $U$ and $V$ are subsets of $\overline \C,$ the notion of a $K$-quasiconformal   mapping $f:U\to V$ is clearly defined (see \cite{AhlforsQuasi}).   We  say that a quasiconformal  mapping $f:\overline\C\to\overline\C$ is normalized if it fixes $\infty,$ $0$ and $1.$ Every quasiconformal  mapping $\overline\C\to\overline\C$ can be normalized by composing it with an appropriate fractional linear transformation.

\begin{lemma}\label{infty}
Every homeomorphism (and in particular every quasiconformal  mapping) $ f:\C\to\C$ extends uniquely to a homeomorphism $\overline f:\overline\C\to\overline\C.$ 
\end{lemma}

We leave the proof to the reader.

The following, which is contained in \cite[Theorem 12.5.3]{AIM} asserts that every quasiconformal  mapping  $\C\to\C$ can be embedded in a holomorphic motion. 

\begin{theorem}\label{embedding}
Let $ h:\C\to\C$ be a quasiconformal  mapping. Then, there exists a holomorphic motion $f:\D\times\C\to\C,$ and a $\lambda\in\D,$ such that $f_\lambda=h.$
\end{theorem}

The following is a refinement of Lemma \ref{infty}.

\begin{corollary}
Every quasiconformal  mapping $ h:\C\to\C$ extends uniquely to a quasiconformal  mapping  $\overline h:\overline\C\to\overline\C.$ 
\end{corollary}

\begin{proof}
The proof follows by first embedding $ h $ in a holomorphic motion and then using the $\lambda$-Lemma.
\end{proof}

 For a subset $X$ of a metric space $M,$  and $d>0,$ denote by $\mathcal H^d(X)$ the $d$-dimensional Hausdorff measure of $X.$ Then, $\mathcal H^d(X)$ is a non-increasing and non-negative function of $d.$  There is a unique number, denoted by $dim(X)$ and called the Hausdorff dimension of $X,$ such that 
$$
	\sup\{d\in[0,+\infty): \mathcal H^d(X)=+\infty\} = \dim(X) = \inf\{d\in[0,+\infty):\mathcal H^d(X)=0\}.
$$

Since the beginning of the  1970s, examples have been known of sets whose Hausdorff dimension is not preserved under quasiconformal mappings.  The following beautiful result of Gehring and  V$\rm{\ddot{a}  is \ddot{a}  l  \ddot{a} }$  is a precise quantitative statement to this effect.  

\begin{theorem} \cite[Cor. 6]{GV}   
For each integer $n\ge 2$ and each pair of numbers $\alpha, \beta\in (0,n),$ there exists a quasiconformal mapping $ h: \R^n\to\R^n$ and a compact set $L\subset\R^n$ such that   
$$
\dim  L = \alpha,   \ \ \  \dim h(L) = \beta.
$$
\end{theorem} 
In particular, we have the following.

\begin{corollary}\label{qcL}	
There is a compact set $L\subset\C$ and a quasiconformal  mapping $h:\C\to\C,$ such that 
$$
	\dim L < 1 \quad \mbox{but} \quad \dim h(L) >1. 
$$	
\end{corollary}

\begin{corollary}\label{holomorphicmotionL}	
There is a compact set $L\subset\C$ and a holomorphic motion $f:\D\times\C\to\C$, such that, for some $\lambda\in\D\setminus\{0\},$
$$
	\dim L < 1 \quad \mbox{but} \quad \dim f_\lambda(L) >1. 
$$	
\end{corollary}

\begin{proof}
This follows immediately from Corollary \ref{qcL} and Theorem \ref{embedding}.
\end{proof}


\section{Dynamics of quadratic polynomials}

In complex dynamics, the dimension of Julia sets  conveys important geometric properties of the dynamics (see \cite{AIM}). Julia sets of quadratic polynomials provide examples of compact sets as in Corollary  {\ref{holomorphicmotionL}}. We say that polynomiall maps $f$ and $g$ from  $\C$ to $\C$
 are {\it conjugate}  if there exists  an affine map $h(\eta)=a+b\eta,$ such that $g=h^{-1}\circ f\circ h.$
 It is  known (see for example \cite[p. 297]{AIM})  that every quadratic polynomial, 
$$
	p(z)=\alpha z^2+\beta z+\gamma,	\quad	\alpha \in\C\setminus \{0\},	\quad	\beta, \gamma \in \C,
$$
is conjugate to a polynomial of the form
\begin{equation}\label{p_c}
	p_c( \zeta)=\zeta^2+c,	\quad	\mbox{where}	\quad	 c\in\C. 
\end{equation}
Consequently, the dynamics of all quadratic polynomials is fully represented by that of the  polynomials $ z^2+c.$

On the other hand, it is easy to see  that, for different values of $c$ the polynomials $p_c$ are never conjugate to each other by  an affine map. Note that the family (\ref{p_c}) of functions from $\overline\C$ to $\overline\C$ is a holomorphic family. 

For each $c$ and each $n=0,1,\ldots,$ denote by $p_c^n$ the $n$th iterate of $p_c.$ 
The {\it Mandelbrot set} is the set of values $c$ for which the orbit of $0$ under iteration by $p^n_c$ is bounded, that is
$$
	\mathcal M = \{c\in\C:  \sup_n|p_c^n(0)|<\infty\}  =  \{c\in\C: p_ c^n(0)    \nrightarrow\infty	\quad	\mbox{as}	\quad	n\to\infty\}.
$$
Clearly, $0$ is in the Mandelbrot set. 



The Mandelbrot set is compact and it follows from the first sentence of the proof of Theorem 1 in the work of John Hubbard and Adrien Douady \cite{DH} that $\overline \C\setminus M$ and $\overline\C\setminus\D$ are conformally equivalent.  Therefore both $M$ and its complement are connected.

A point $x\in\overline\C$ is {\it stable}  for $p_c$ if on some neighbourhood $V$ of $x$ the family 
$$
	p^n_c:V\to\overline\C, \quad  n=0,1,2,\ldots
$$
is equicontinuous.  
The set $\mathcal F_c=\mathcal F(p_c)$ of stable points is called the {\it Fatou set} of $p_c.$ 
The set $J_c=J(p_c)$ of {\it unstable} points (i.e. points that are not stable) is called the {\it Julia set} of $p_c$.  The Julia set also has the following characterization: 
$$
	J_c=\partial\{z: p_c^n(z)\to\infty\} =\partial\{z: p_c^n(z)\nrightarrow\infty\}.
$$  
The Julia sets $J_c$ are  compact, non-empty, with empty interior 
and for $c$ outside of the Mandelbrot set,  they  are Cantor sets and hence are all topologically equivalent.

We now provide some basic and well-known results on quadratic dynamics  which we shall employ.

From an estimate of Thomas Ransford \cite[Example 2]{R}, we have the following. 
\begin{theorem}\label{Ransford}
$$
	\dim J_c\to 0	\quad	\mbox{as}		\quad		c\to\infty.
$$
\end{theorem}

The next theorem is due to Mitsuhiro Shishikura.\\

\begin{theorem}\cite[Remark 1.1]{S}\label{Shishikura}
$$
	\sup_{c\in \C\setminus\mathcal M} \dim J_c	=2.
$$
\end{theorem}

These two theorem confirm the  well-known fact  that, for some $c\in\C\setminus\mathcal M,$  the Hausdorff dimension of $J_c$ can be arbitrarily small and, for $c\in\C\setminus\mathcal M,$  the Hausdorff dimension of $J_c$ can be arbitrarily  close to $2.$

The following lemma provides more information regarding the Julia sets $J_c$ for $c$ outside of the Mandelbrot set.  

\begin{lemma}\label{distance}
The Julia set $J_c$ varies continuously with respect to the Hausdorff distance, as $c$ varies in the complement $\C\setminus\mathcal M$ of the Mandelbrot set. 
\end{lemma}

\begin{proof}

A short proof is simply to refer to \cite{MSS}.
Actually \cite{MSS} proves a much stronger statement about ''structural stability'' of rational maps.  One must use the fact that a critical orbit does not change its behavior (escaping to infinity) in the complement of the Mandelbrot set.  
The statement of the lemma has been known as folklore.  There is a general theory of structural stability of (not necessarily holomorphic) dynamical systems.
Starting from the Smale school in the west, (and there was a counter part in the Soviet Union), it has been well known that when a map is hyperbolic (in the dynamical sense) or expanding (in the case of complex dynamics), the dynamical behavior does not change when you perturb it.  So before the paper \cite{MSS}, there was a wide consensus that an expanding map (for example, $z^2+c$ with $c$ in the complement of M) is structurally stable.  

\end{proof}

Combining Lemma \ref{distance} with \cite[Theorem 7.2]{M} we have that the Julia set $J_c $   moves locally in $\C\setminus\mathcal M$ by a conjugating holomorphic motion.  
 It then follows from  \cite{MSS}, that
Julia sets $J_c,$ where $c$ lies outside of the Mandelbrot set, are locally isotopic by a holomorphic motion.  

Hence, by the $\lambda$-lemma (Theorem \ref{th:lambda}), they are  locally quasiconformally equivalent 
and since $\C\setminus\mathcal M$ is connected, they are all quasiconformally equivalent. Now, invoking Theorems \ref{Ransford} and \ref{Shishikura},
this provides another way of obtaining a compact set $L$ satisfying Corollary \ref{qcL} and Corollary \ref{holomorphicmotionL}.


\section{Polynomial and rational  approximation}

For an introduction to complex approximation in one variable, we recommend Gaier's book \cite{G} and for a recent overview of uniform complex approximation in both one and several variables, see the excellent survey by Forn\ae ss, Forstneri$\rm{\check{c}}$ and Wold
\cite{FFW}.

When we say that a function is holomorphic on a set $ K \subset\C,$ we mean that it is holomorphic on an open set containing $ K.$ For a compact set $K\subset\C,$  $P(K)$ denotes the functions on $K$ which are uniform limits on $K$  of polynomials and  $R(K)$ the functions on $K$ which are uniform limits of  rational functions having no poles on $K.$   We denote by $A( K)$ the class of functions  which are continuous on $K$ and holomorphic on the interior $ K^\circ$ of $ K.$

For every compact set $K\subset\C,$ 
\begin{equation*}\label{inclusions}
P(K)\subset R(K)\subset   A(K) \subset C(K).
\end{equation*}
These inclusions are in general strict and the fundamental problem in complex approximation is, for each of these inclusions, to characterize those sets for which we have equality. 

Let us say that $K$ is a set of polynomial approximation if $P(K)=A(K).$ 
Mergelyan's theorem (see \cite{G}) gives a topological characterization for set to be a set of polynomial approximation.  Namely, $K$ is a set of polynomial approximation if and only if its complement $\C\setminus K$ is connected.     

Let us say that $K$ is a set of rational approximation if $R(K)=P(K).$ A characterization for sets of rational approximation was given by Vitushkin (see \cite{G}) in terms of continuous analytic capacity. There is no topological characterization for sets of rational approximation (see \cite{GHL}).  That is, there exists a homeomorphism $h:\C\to\C,$ such that $R(K)=A(K)$ but $R\big(h(K)\big)\not= A\big(h(K)\big).$ Thus,  the property of being a compact set of rational approximation is not topologically invariant. Our main objective is to show  that it is also not invariant under quasiconformal transformations and consequently not invariant under holomorphic motion.


\section{A holomorphic motion which does not preserve rational approximation}

Let us denote by $\partial_iK$ the {\it inner boundary} of a compact set $K\subset\C,$ that is, the set of all boundary points which are not boundary points of some complementary component.

\begin{theorem} [ Davie-{\O}ksendal \cite{DO}] \label{inner boundary}
For a compact set $K\subset\C,$ 
if $\dim\partial_i K<1,$ then $R(K)=A(K).$ 
\end{theorem}

Let us recall (see \cite{G}) the definition of the continuous analytic capacity $\alpha(M)$ of a set $M\subset \C.$ Denote by $X(M)$ the class of functions $f$ that are continuous on $\C$ satisfying $f(z)\to 0$ as $z\to \infty,$ are bounded by $1$ and are holomorphic outside of a compact subset of $M.$ In a neighborhood of $\infty$ such a function has an expansion 
$$
	f(z) = \sum_{n=1}^\infty\frac{c_n(f)}{z^n}
$$
and   $f^\prime(\infty)$ is defined as follows:
$$
	f^\prime(\infty)=c_1(f)=\lim_{z\to\infty}zf(z) = \frac{1}{2\pi i}\int_{|z|=r}f(z)dz, \quad \mbox{for large $r$}.
$$
Then, $\alpha(M)=\sup\{|c_1(f)|: f\in X(M)\}.$

 For ``large" Hausdorff dimension, we have the following (see for example  \cite[Page 229, lines 7,8]{Y}. 

\begin{lemma}\label{alpha}
For a  set $L\subset\C,$ if dim$(L)>1,$ then $\alpha(L)>0$.
\end{lemma}

\bigskip
\noindent
\begin{theorem} \label{th:mainnotinvariant}
There is a compact set $K\subset\C$ and a  quasiconformal  mapping $h:\C\to\C$ such that 
$$
	R(K)=A(K) \quad \mbox{but} \quad R(h(K))\not=A(h(K)). 
$$
 Thus, rational approximation is not invariant under quasiconformal  mappings. 

\end{theorem}
 
\begin{proof}

To give such an example, 
let $L$ and $h$ be as in {Corollary \ref{qcL}} and denote $\widetilde L = h(L).$ 
We may suppose that both $L$ and $\widetilde L$ are contained in the open unit disc.
Note that $\widetilde L$ is a compact subset of the open unit disc having no interior.
Let  $\widetilde K=\D\setminus \cup_j D_j,$  where the $D_j$ are open discs,  whose closures are disjoint and  contained in $\D\setminus\widetilde L,$ such that 
 $\widetilde L$ is the set of accumulation points of the sequence $\{a_j \}_{j=1}^\infty$ of centers of the $D_j$'s.  Thus $\widetilde L$ is the inner boundary of $\widetilde K$. The discs are chosen so small that  $\sum 2\pi r_j< 2\pi\alpha(\widetilde L)/2,$  where the $r_j$'s are the respective radii of the $D_j$'s. Set $K=h^{-1}(\widetilde K).$ Then $K$ is a compact set with inner boundary $L.$ Thus, 
$A(K)=R(K)$  by Theorem \ref{inner boundary}, since $dim (L)<1.$
 
There remains to show that  $R(\widetilde K)\not=A(\widetilde K).$
Since $dim(\widetilde L)>1,$ it follows from Lemma \ref{alpha} that $\alpha(\widetilde L)>0$, so we can find $f \in C(\mathbb{C})$, 
holomorphic outside $\widetilde L$ and bounded by $1,$ with $f(\infty)=0$ and $|f^\prime(\infty)|> \alpha(\widetilde L )/2,$ where
$f'(\infty)= \lim_{z\to \infty} (zf(z))=(1/2\pi i)\int_{\partial\D} f(z)dz \neq 0$. 
We claim that $f\in A(\widetilde K)\setminus R(\widetilde K).$ Clearly, $f\in A(\widetilde K),$ since $\widetilde K^\circ\cap \widetilde L=\emptyset.$ 

To show that $R(\widetilde K)\not=A(\widetilde K)$,  
it is sufficient to find a Borel measure $\mu$ on $\widetilde K,$ which is orthogonal to $R(\widetilde K)$ but not orthogonal to $A(\widetilde K).$ We claim that  $\mu=dz_{\partial\D} -\sum_j  dz_{\partial D_j}$ is such a measure. 

Let $g$ be a rational function having no poles on $\widetilde K.$ By renumbering the $D_j$'s, we may assume that $D_1,\ldots, D_k$ are the only $D_j$'s containing poles of $g.$ By the Cauchy theorem $\int_{\partial D_j}g(z) dz=0,$  for $j>k.$ Hence 
$$
	\int g(z)d\mu(z) = \int_{\partial \D}g(z)dz - \sum_{j=1}^k\int_{\partial D_j}g(z)dz = 0,
$$
because  $g$ is holomorphic on the domain bounded by $\partial D$ and $\partial D_j, \, j=1,\ldots,k.$ 
Thus $\mu$ is orthogonal to $R(\widetilde K).$ 

To see that $\mu$ is not orthogonal to $A(\widetilde K),$ it suffices to check that  
 $\int f(z)d\mu (z) \not=0.$ Indeed, 
$$
	\left|\int f(z)d\mu(z)\right| \ge \left|\int_{\partial\D} f(z)dz\right| - \sum_{j=1}^\infty 2\pi r_j >
2\pi\alpha(\widetilde L)/2-2\pi \alpha(\widetilde L)/2=0.
$$
Thus $f \in A(\widetilde K)\setminus R(\widetilde K)$ because the measure $\mu=dz_{\partial\D} -\sum_j  dz_{\partial D_j}$
is orthogonal to $R(\widetilde K)$, but not to $ f$. 
This concludes the validation of Theorem \ref{th:mainnotinvariant}. 
\end{proof}

\begin{theorem}
There is a compact set $K\subset\C,$  a holomorphic motion of $K,$  $H:\C\times K\to\C$ and a $\lambda\in\D,$  such that 
$$
	R(K)=A(K) \quad \mbox{but} \quad R(H_\lambda(K))\not=A(H_\lambda(K)). 
$$
Thus, rational approximation is not invariant under holomorphic motion. 
\end{theorem}

\begin{proof}
Let $K$ and $h$ be as in Theorem \ref{th:mainnotinvariant}. By Theorem \ref{embedding}, there is a holomorphic motion $H:\D\times\C\to\C$ and a $\lambda\in\D,$ such that $H_\lambda=h.$  The restriction of the holomorphic motion $H$ to $\D\times K$ is the desired holomorphic motion of $K.$ 
\end{proof}

{\it Remark.} Holomorphic motion is a strong form of isotopy. In 1980 Scheinberg
 \cite{Scheinberg} used quasiconformal  mappings to show the non-invariance under analytic isotopy of a certain approximation property for closed subsets of Riemann surfaces.   The approximation property he considered has no bearing on rational approximation  and his techniques are different from ours.


\section{Holomorphic motion and the Riemann Hypothesis}

In a fundamental paper in  1987, Bhaskar Bagchi \cite{B1987}  presented a statement equivalent to the Riemann Hypothesis in terms of topological dynamics, showing that the Riemann Hypothesis holds if and only if the Riemann zeta-function satisfies the conclusion of the Voronin Universality Theorem. In 2018, Tomoki Kawahira \cite{K2018} gave a version of the Riemann Hypothesis in terms of complex dynamics.  In the present section, we show that the Riemann Hypothesis holds if and only if a certain explicit mapping is a holomorphic motion.

Instead of considering holomorphic motions   parametrized by  the disc $\D,$ one can consider holomorphic motion   parametrized by  a complex  manifold $\Omega,$ and since the arithmetic properties of $0$ do not enter into the definition of holomorphic motion, we may replace  the point 0  by a
a distinguished point  $\lambda_0\in\Omega$  playing the role of $0$ in $\D.$ For example, in \cite{Chirka2004}, Chirka considers holomorphic motion   parameterized by  a Riemann surface $\Omega$.

\begin{lemma}\label{wierd}
Let $\Omega$ be a complex manifold and $\lambda_0$ a distinguished point of $\Omega.$ Let $E\subset\C$ and $\varphi:\Omega\to\overline\C$ be a holomorphic mapping such that $\varphi(\lambda_0)=0$ and consider
$$
	f:\Omega\times E\to \overline\C, \quad	(\lambda,z)\mapsto z+\varphi(\lambda). 
$$	
The following are equivalent

a)  $\varphi(\lambda)\not=\infty, $	for all $\lambda\in\Omega;$

b) $f$ is a holomorphic motion;

c) $f$ is a holomorphic motion in $\C.$
\end{lemma}

\begin{proof} 
The mapping $f$ is a holomorphic motion if and only if it satisfies the following conditions. 
\begin{enumerate}
		\item  $f(\lambda_0,z)=z,$   for all $z\in
		E$; 
		\item For
		every fixed $z\in E$, $f(\cdot,z):\Omega\to\overline\C$ is a holomorphic map;
		\item For every fixed $\lambda\in \Omega$, 
		$f(\lambda,\cdot):E\to\overline\C$ is an injection. 
	\end{enumerate}
Conditions (1) and (2) are satisfied, 
so $f$ is a holomorphic motion if and only if (3) holds. 

a)$ \to$ b) Suppose a).  Then,  for every $\lambda\in\Omega,$ $\varphi(\lambda)\not=
\infty,$ so (3) is satisfied. Thus, $f$ is a holomorphic motion.

b)$\to$ a) Suppose b).  Then (3) holds, so $\varphi(\lambda)\not=\infty,$ for all $\lambda\in\Omega,$ which implies a). 

c)$\to$ b) This is trivial.  

a)$\to$ c) If we have a) we have shown that we have b), so $f(\lambda,z)=z+\varphi(\lambda)$ is a holomorphic motion. Since we have a),  $f$   takes its values in $\C.$ 
\end{proof}

Now, we consider the particular case where $\Omega$ is  the strip $\Omega = \{\lambda\in\C:1/2<\Re \lambda < 3/2\},$ with distinguished point $\lambda=1.$ Since there is a  conformal map from $\D$ to $\Omega$ which sends $0$ to $1,$ holomorphic motion on the strip $\Omega$ is equivalent to standard holomorphic motion on the disc $\D.$ 
Consider the mapping: 
$$
	f:\Omega\times\C\to\overline\C, \quad (\lambda,z)\mapsto z+1/\zeta(\lambda),
$$
where $\zeta(\lambda)$ is the Riemann zeta function.
\begin{theorem}
The following are equivalent:

a) The Riemann hypothesis;

b) $f$ is a holomorphic motion;

c) $f$ is a holomorphic motion in $\C.$
\end{theorem}


\begin{proof}
Set $\varphi(\lambda)=1/\zeta(\lambda).$ By Lemma \ref{wierd} it is sufficient to verify that the Riemann hypothesis holds if and only if $\varphi(\lambda)\not=\infty,$ for all $\lambda\in \Omega,$ which is equivalent to the statement that $\zeta(\lambda)\not=0,$ for $1/2<\Re\zeta<3/2,$ which is equivalent to the statement that $\zeta(\lambda)\not=0,$ for $1/2<\Re\zeta<1.$ Since, the zeros of $\zeta$ in the fundamental strip $0<\Re\zeta<1$ are symmetric with respect to the critical axis $\Re\zeta=1/2,$ the latter statement is equivalent to the statement that the zeros of $\zeta$ in the fundamental strip all lie on the critical axis, which is the Riemann hypothesis.  
\end{proof}


%

{\bf Acknowledgement.}  The second author expresses great gratitude to  Bo-Yong Chen for introducing  the topic of holomorphic motion to her and for his continuous encouragement and support. We thank Petr Vladimirovich  Paramonov and Xavier Tolsa, who generously responded to our request for help regarding rational approximation. We thank Malik Younsi and Thomas Ransford who introduced us to the relation between dimension and complex dynamics. We thank Takeo Ohsawa for bringing to our attention the work of Tomoki  Kawahira on complex dynamics and the Riemann Hypothesis. 
We also thank John Erik Forn{\ae}ss for helpful comments.




\begin{thebibliography}{11}
\bibitem{AhlforsQuasi}Ahlfors,  L. ,  Lectures on Quasi conformal Mappings, University Lecture Series 038, American Mathematical Society,  2006.


 
\bibitem{AIM}
Astala, K.,   Iwaniec, T.  and Martin,  G. ,  Elliptic partial differential equations and quasiconformal mappings in the plane,
Princeton Math. Ser., {\bf 48}
Princeton University Press, Princeton, NJ, 2009.


 


\bibitem{B1987} 
Bagchi, B., 
 Recurrence in topological dynamics and the Riemann hypothesis,
Acta Math. Hungar. {\bf 50} (1987), no. 3-4, 227-240.

 

\bibitem{BersRoyden} Bers,  L.  and Royden,  H. L. ,   Holomorphic families of injections, Acta Math. {\bf 157} (1986), 259--286.	


 
\bibitem{Chirka2004} Chirka, E. M.,
On the propagation of holomorphic motions.(Russian) Dokl. Akad. Nauk {\bf 397}  (2004), no.1, 37-40.


\bibitem{DO}  Davie,  A.M. and {\O}ksendal,  B.,  Analytic capacity and differentiability properties of finely harmonic functions.  Acta Math. 149 (1982), no. 1-2, 127-152.



\bibitem{DH} Douady, A. and Hubbard, J. H., It\'{e}ration des polyn\^{o}mes quadratiques complexes,
C. R. Acad. Sci. Paris S\'{e}r. I Math. {\bf 294} (1982), no. 3, 123-126.




\bibitem{FFW} Forn\ae ss,  J. E,  Forstneri$\rm{\check{c}}$,  F.  and Wold,  E. F.,  Holomorphic approximation: the legacy of Weierstrass, Runge, Oka-Weil, and Mergelyan. Advancements in complex analysis-from theory to practice, 133-192.
Springer, Cham, [2020].

\bibitem{G} Gaier, D.,  Lectures on complex approximation. Translated from the German by Renate McLaughlin. Birkh\"auser Boston, Inc., Boston, MA, 1987.


\bibitem{GHL} Gauthier, P. M., Hengartner, W. and Labr$\rm{\grave{e}}$che, M.,
Approximation harmonique, approximation holomorphe et topologie. (French) [Harmonic approximation, holomorphic approximation and topology]
Canadian J. Math. {\bf 34} (1982), no. 1, 216--219.

\bibitem{GJW2010} Gardiner, F.P.,   Jiang, Y.  and Wang, Z. ,  Holomorphic motions and related topics,  arXiv:0802.2111v1.,  https://doi.org/10.48550/arXiv.0802.2111. 


 
\bibitem{GV} Gehring, F. W. and V$\rm{\ddot{a}}$is$\rm{\ddot{a}}$l$\rm{\ddot{a}}$, J. ,
Hausdorff dimension and quasiconformal mappings
J. London Math. Soc. (2) {\bf 6} (1973), 504-512.
 
 

\bibitem{K2018} Kawahira, T., 
 The Riemann hypothesis and holomorphic index in complex dynamics,
Exp. Math. {\bf 27} (2018), no. 1, 37-46.

 


\bibitem{LV} 
Lehto, O. and Virtanen, K. I.,
Quasiconformal mappings in the plane,
Die Grundlehren der mathematischen Wissenschaften, Band {\bf 126},
Springer-Verlag, New York-Heidelberg, 1973.



\bibitem{MSS} Ma$\tilde{\text{n}}\acute{\text{e}}$,  R. ,  Sad, P.  and Sullivan,  D.,
{On the dynamics of rational maps}, Ann. Sci. Ecole
Norm. Sup. {\bf 16} (1983), 193--217.



\bibitem{M} McMullen, C., Riemann surfaces, dynamics and geometry. Course Notes,  October 7, 1920, Harvard University. 




 
\bibitem{R} Ransford, T., Variation of Hausdorff dimension of Julia sets,
Ergodic Theory Dynam. Systems {\bf 13} (1993), no. 1, 167-174.
 

\bibitem{Scheinberg}Scheinberg,  S.,
Non-invariance of an approximation property for closed subsets of Riemann surfaces.
Trans. Amer. Math. Soc. {\bf 262} (1980), no. 1, 245-258.


\bibitem{S} Shishikura, M.,
The Hausdorff dimension of the boundary of the Mandelbrot set and Julia sets.
Ann. of Math. (2) {\bf 147} (1998), no. 2, 225-267.

	
\bibitem{Slodkowski} S{\l}odkowski,  Z., {Holomorphic motions and polynomial hulls}, Proc. Amer. Math. Soc. {\bf 111} (1991), 347--355.



\bibitem{SullivanThurston} Sullivan, D. and Thurston,  W. P. , {Extending holomorphic motions}, Acta Math. {\bf 157} (1986), 243--257.
 



\bibitem{Y} Younsi, M.,   On removable sets for holomorphic functions. EMS Surv. Math. Sci. {\bf 2} (2015), no. 2, 219-254.

 
\end{thebibliography}
\end{document}